\documentclass[12pt,a4paper]{article}
\usepackage{color}
\usepackage{graphicx}
\usepackage{tabularx}
\usepackage{amsmath,amssymb,mathrsfs,wasysym}
\usepackage{amsthm}
\usepackage{multicol}
\usepackage{enumerate}
\usepackage{float}
\usepackage{makeidx}  
\date{}

\newtheorem{theorem}{Theorem}

\newtheorem{prop}{Proposition}
\newtheorem{cor}{Corollary}
\theoremstyle{definition}
\newtheorem{defn}{Definition}
\newtheorem{rem}{Remark}

\begin{document}
\title{\bf  The distant graph of the projective line over a finite ring with unity}
\author{Edyta Bartnicka, Andrzej Matra{\'s}}
\maketitle
\begin{abstract}
\noindent
We discuss the projective line $\mathbb{P}(R)$ over a finite associative ring with unity. $\mathbb{P}(R)$ is naturally endowed with the symmetric and anti-reflexive relation "distant". We study the graph of this relation on $\mathbb{P}(R)$ and classify up to isomorphism all distant graphs $G(R, \Delta)$ for rings $R$ up to order $p^5$, $p$ prime.
\end{abstract}

\section{Introduction}
The aim of this paper is to characterize the distant graph $G(R, \Delta)$ of the projective line over any finite ring $R$. It is an undirected, connected graph with the degree of a vertex equal to $|R|$.\newline 
The starting point of our investigation is showing the connection between this graph and the distant graph $G(R/J, \Delta_J)$ of the projective line over the factor ring $R/J$, where $J$ is the Jacobson radical of $R$. To this end we use, introduced by Blunck and Havlicek in \cite{radical}, an equivalence relation, called radical parallelism, on the set of points of the projective line, which determines the interdependence between $\mathbb{P}(R)$ and $\mathbb{P}(R/J)$. Next we describe the graph $G(R/J, \Delta_J)$. Using structures theorems \cite{behrens} on finite rings with unity we get that the graph $G(R/J, \Delta_J)$ is isomorphic to the tensor product of the distant graphs arising from projective lines whose underlying rings are full matrix rings over finite fields. The projective line over any full matrix ring $M_n(q)$, i.e. the ring of $n\times n$ matrices over the finite field $F(q)$ of order $q$, is in bijective correspondence with the Grassmannian $\mathscr G(n,2n,q)$ of $n$-dimensional subspaces of a $2n$-dimensional vector  space over $F(q)$. Then we describe $G(M_2(q), \Delta)$ for any prime power $q$ and we give representatives of two classes of partitions of $G(M_2(2), \Delta)$  on a sum of vertex-disjoint maximal cliques. We also make use of these partitions to show a simple construction of projective space of order $2$, described by Hirschfeld in \cite{hir} in a completely different way. The question still unanswered is, whether a partition of $G\big(M_n(q), \Delta)\big)$ on a sum of vertex-disjoint maximal cliques exists for any $n, q$. However, this partition of $G(R, \Delta)$ for any finite ring R such that $R/J$ is isomorphic to the
direct product of n copies of $F(q)$ is done in the present study, in particular the distant graph of the projective line over any ring of lower triangular matrices over $F(q)$ .\newline
Using the classification of finite rings from \cite{p4}, \cite{p5}, we find all nonisomorphic distant graphs $G(R, \Delta)$ for rings $R$ up to order $p^5$, $p$ prime, in the last section. We also describe the graph $G(R, \Delta)$ in the case of an arbitrary local ring $R$.

\section{Preliminaries}
Throughout this paper we shall only study finite associative rings with 1 ($1\neq 0$). Consider the free left module  $^2R$ over a ring $R$. Let  $(a, b)\in {^2R}$, the set $$R(a, b)=\{(\alpha a, \alpha b); \alpha \in R\}$$ is a {\sl left cyclic submodule of ${^2R}$}. If the equation $(ra, rb)=(0, 0)$ implies that $r=0$, then $R(a, b)$ is called {\sl free}. A pair $(a, b)\in {^2R}$ is called {\sl admissible}, if there exist elements $c, d \in R$ such that $$\left[\begin{array}{cclr}
a&b\\ c&d
\end{array}\right] \in GL_2(R).$$ 
The general linear group $GL_2(R)$ acts in natural way (from the right) on the free left $R$-module $^2R$ and this action is transitive.
\begin{defn}\label{p.line}{\rm \cite{hav1}}
The {\sl projective line over $R$} is the orbit $$\mathbb{P}(R):=R(1, 0)^{GL_2(R)}$$ of the free cyclic submodule $R(1, 0)$ under the action of $GL_2(R)$.
\end{defn}
In other words, the points of $\mathbb{P}(R)$ are those free cyclic submodules $R(a, b)\in {^2R}$  which possess a free cyclic complement, i.e. they are generated by admissible pairs $(a, b)$.

We recall that a pair $(a, b)\in {^2R}$ is {\sl unimodular}, if there exist $x, y\in R$ such that $$ax+by=1.$$
It is known that if $R$ is a ring of stable rank 2, then admissibility and unimodularity are equivalent and $R$ is Dedekind-finite \cite[Remark 2.4]{hav1}. Rings that are finite or commutative satisfy this property, so in case of such rings, the projective line can be described by using unimodular or admissible pairs interchangeably.\newline
A wealth of further references is contained in \cite{chain.g}, \cite{kettengeom}.

\begin{defn}\label{distant}{\rm\cite{radical}}
The point set $\mathbb{P}(R)$ is endowed with the symmetric and anti-reflexive relation {\sl distant} which is defined via the action of $GL_2(R)$ on the set of pairs of points by 
$$\Delta:=\bigl(R(1, 0), R(0, 1)\bigr)^{GL_2(R)}.$$\end{defn}

It means that $$\bigwedge_{ R(a, b), R(c, d)\in \mathbb{P}(R)} R(a, b)\Delta R(c, d)\Leftrightarrow \left[\begin{array}{cclr}a & b \\ c & d \end{array}\right]\in GL_{2}(R).$$
\\
Moreover, $$\bigwedge_{ \varphi \in GL_{2}(R)}\ \ \ \bigwedge_{ R(a, b), R(c, d)\in \mathbb{P}(R)} R(a, b)\Delta R(c, d)\Leftrightarrow  (R(a, b))^{\varphi} \Delta (R(c, d))^{\varphi}.$$
The next relation on $\mathbb{P}(R)$ is connected with the Jacobson radical of $R$, denoted by $J$. It is that two-sided ideal which is the intersection of all the maximal right (or left) ideals of $R$.\\
Namely, in \cite{radical} A. Blunck and H. Havlicek introduced an equivalence relation in the set of pairs of non-distant points called {\sl radical parallelism} $(\parallel)$ as follows: $$R(a, b)\parallel R(c, d)\Leftrightarrow \Delta(R(a, b))=\Delta(R(c, d)),$$ where $\Delta(R(a, b))$ is the set of those points of $\mathbb{P}(R)$ which are distant to $R(a, b) \in \mathbb{P}(R)$. In this case we say that a point $R(a, b)\in \mathbb{P}(R)$ is {\sl radically parallel} to a point $R(c, d)\in \mathbb{P}(R)$.\\
The canonical epimorphism $R\to R/J$ sends any $a\in R$ to $a\mapsto
a+J=:\overline a$. According to (\cite[Theorem 2.2]{radical}) the mapping
\begin{equation*}
    \Phi: R(a, b) \mapsto R/J(\overline a, \overline b)
\end{equation*}
is well defined and it satisfies
\begin{equation*}
    R(a,b)\parallel R(c,d) \Leftrightarrow R/J(\overline{a}, \overline{b})=R/J(\overline{c}, \overline{d}).
\end{equation*}

\begin{rem}\cite{radical}\label{ilorazowy}
Furthermore, we have
$$\bigwedge_{a, b\in R}R(a, b)\in \mathbb{P}(R)\Leftrightarrow R/J(\overline{a}, \overline{b})\in \mathbb{P}(R/J)$$
and $$\bigwedge_{R(a, b), R(c, d)\in \mathbb{P}(R)}R(a, b)\Delta R(c, d)\Leftrightarrow R/J(\overline{a}, \overline{b})\Delta_J R/J(\overline{c}, \overline{d}),$$
where $\Delta_J$ denotes the distant relation on $\mathbb{P}(R/J)$.
\end{rem}

Therefore, the radical parallelism relation determines the connection between projective lines $\mathbb{P}(R)$ and $\mathbb{P}(R/J)$. \newline\\
Since the point set $\mathbb{P}(R)$ is endowed with the distant relation, we can consider $\mathbb{P}(R)$ as the set of vertices $V\bigl(G(R, \Delta)\bigr)$  of the {\sl distant graph} $G(R, \Delta)$, i.e. the undirected graph of the relation $\Delta$. 
Its vertices are joined by an edge if, and only if, they are distant. This graph is connected
and its diameter is less or equal 2 \cite[1.4.2. Proposition]{chain.g}.\newline
One of the basic concepts of graph theory is that of a clique. A {\sl clique} in an undirected graph $G$ is a subset of the vertices such that every two distinct vertices comprise an edge, i.e. the subgraph of $G$ induced by these vertices is complete. A {\sl maximum clique} of a graph $G$ is a clique, such that there is no clique  in $G$ with more vertices. A {\sl maximal clique} is a clique which is not properly contained in any
clique. All maximal cliques in $G(R, \Delta)$ has the same number of vertices, denoted by $\omega\bigl(G(R, \Delta)\bigr)$, and at the same time, they are maximum cliques.\newline 
To describe maximal cliques of the distant graph we make use of the following definitions. 

\begin{defn}
An {\sl $(n-1)$-spread} in the $(2n-1)$-dimensional projective space $PG(2n-1, q)$ over the finite field  with $q$ elements $F(q)$  is a set of $(n-1)$-dimensional subspaces such that each point of $PG(2n-1, q)$ is contained in exactly one element of this set.
\end{defn}

\begin{defn}
An {\sl $(n-1)$-parallelism} in $PG(2n-1, q)$ is a partition of the set of $(n-1)$-dimensional subspaces of $PG(2n-1, q)$ by pairwise disjoint $(n-1)$-spreads.
\end{defn}

\begin{rem}\label{V-PG}
$PG(2n-1, q)$ corresponds to a $2n$-dimensional vector space $V(2n, q)$ over the finite field $F(q)$ , and  $(n-1)$-dimensional subspaces of $PG(2n-1, q)$ correspond to $n$-dimensional subspaces of $V(2n, q)$. Consequently, we can also talk about $n$-spreads and $n$-parallelisms of $V(2n,q)$ rather than $(n-1)$-spreads and $(n-1)$-parallelisms of $PG(2n-1,q)$.
\end{rem}

In our considerations we will often use the following fact about the direct product of projective lines.
\begin{theorem} \label{product}{\rm \cite[6.1.]{hav3}}
Let $R$ be the direct product of rings $R_i, \ i=1, 2, \ldots, n$, i.e. $R=R_1\times R_2\times \cdots \times R_n$. 
Then
$$
\mathbb{P}(R, \Delta)\simeq \mathbb{P}(R_1, \Delta_1)\times \mathbb{P}(R_2, \Delta_2)\times \cdots \times \mathbb{P}(R_n, \Delta_n),
$$
where $\Delta_i$ stands for the distant relation on $\mathbb{P}(R_i)$.
\end{theorem}

In \cite{ja} was pointed out that in another way to state this is to say
$$
G(R, \Delta)\simeq G(R_1, \Delta_1)\times G(R_2, \Delta_2)\times \cdots \times G(R_n, \Delta_n),
$$ 
what means that the graph $G(R, \Delta)$ is the tensor product of the graphs\break $G(R_1, \Delta_1), G(R_2, \Delta_2), \ldots, G(R_n, \Delta_n)$, i.e., the vertex set of $G(R, \Delta)$ is the Cartesian product of $G(R_1, \Delta_1), G(R_2, \Delta_2), \ldots, G(R_n, \Delta_n)$, and for all\break $(x_1, x_2, \ldots, x_n), (x'_1, x'_2, \ldots, x'_n)\in G(R, \Delta)$ holds
$$(x_1, x_2, \ldots, x_n)\Delta (x'_1, x'_2, \ldots, x'_n)\Leftrightarrow x_1\Delta_1 x'_1, x_2\Delta_2 x'_2, \ldots, x_n\Delta_n x'_n.$$ 

\section{Construction of the distant graph on the projective line}
\label{construction}

In order to describe the distant graph $G(R, \Delta)$ of the projective line over a ring $R$ we show the connection between this graph and the distant graph $G(R/J, \Delta_J)$ of the projective line over the factor ring $R/J$. Next we find the graph $G(R/J, \Delta_J)$.\newline
The points of $\mathbb{P}(R/J)$ are in one-one correspondence with the equivalence classes of the radical parallelism relation on $\mathbb{P}(R)$. Each of these comprises $|J|$ elements. See \cite{radical} for more details.
Write $$\overline{a}=\{a_i; i=1, 2, \ldots, |J|\}$$ for all $a\in R$. For any point $R/J(\overline{a}, \overline{b})\in \mathbb{P}(R/J)$ there exist exactly $|J|$ different points $R(a_i, b_i)\in \mathbb{P}(R)$ such that $R/J(\overline{a_i}, \overline{b_i})=R/J(\overline{a}, \overline{b})$. 
Then the graph $G(R, \Delta)$ is uniquely determined by the Remark \ref{ilorazowy}.
For example, if $R=T(2)$ is the ring of ternions over the field $F(2)$, then the projective line over $T(2)/J$ has a distant graph which is depicted in Figure \ref{a} (left). For better visualisation we only show the vertices and the edges of the graph $G(T(2), \Delta)$ corresponding to those marked with a line in $G(T(2)/J, \Delta_J)$.\break
\begin{figure}[hhh]
\centerline{\includegraphics[width=0.70\textwidth]{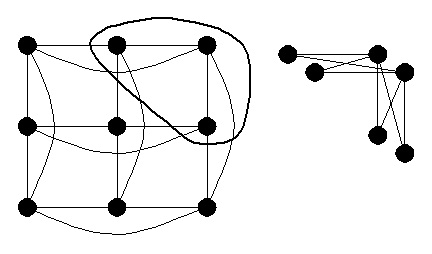}}
\caption{The connection between $G\big(T(2)/J, \Delta_J\big)$ and  $G\big(T(2), \Delta\big)$}
\label{a}
\end{figure}

\begin{prop}\label{kliki.iloraz} 
$G(R/J, \Delta_J)$ has a partition on a sum of $m$ vertex-disjoint maximal cliques if, and only if, $G(R, \Delta)$ has a partition on a sum of $m|J|$ vertex-disjoint maximal cliques.  For these partitions the following equality holds
$$\omega\bigl(G(R/J, \Delta_J)\bigr)=\omega\bigl(G(R, \Delta)\bigr).$$
\end{prop}
\begin{proof}
This follows from Remark \ref{ilorazowy} and the the fact that $\Phi^{-1}\bigl(R/J(\overline a, \overline b)\bigr)$ containes exactly $|J|$ points  for any $(R/J(\overline a, \overline b)$.
\end{proof}

\begin{prop}\label{kliki.produkt}
Let $G_i$ be a graph such that $V(G_i)$ is a sum of vertex-disjoint maximal cliques $K_{t_i}, t_i=1, \ldots, m_i$, of the same cardinality, i.e. $|K_{t_i}|=s_i$, for all $i=1, \ldots, n$. Write $\min\{s_i; i=1, \ldots, n\}=s$. 
The graph $\bigotimes_{i=1}^nG_i$ has a partition on a sum of $\frac{m_1 \cdots m_ns_1\cdots s_n}{s}$ vertex-disjoint maximal cliques with $\omega(\bigotimes_{i=1}^nG_i)=s$.
\end{prop}
\begin{proof}
Suppose that the above assumptions are satisfied. 
We have $$V(G_1)\times\cdots\times V(G_n)=
\bigcup_{t_1=1}^{m_1}K_{t_1}\times\cdots\times \bigcup_{t_n=1}^{m_n}K_{t_n}=\bigcup_{t_1=1}^{m_1}\ldots
\bigcup_{t_n=1}^{m_n}K_{t_1}\times\cdots\times K_{t_n}.$$
We first give the proof for the case $n=2$. Without loss of generality, assume that $s_1\leqslant s_2.$ 
Let $(k_{l_1}, k_{l_2}),\ l_1=1, \ldots, s_1,\ l_2=1, \ldots, s_2$, be  vertices of $K_{t_1}\times K_{t_2}$. 
For all ${t_1}=1,\ldots, m_1$, ${t_2}=1,\ldots, m_2$, $K_{t_1}\times K_{t_2}$ is equal to the set\newline\\
$\left\{\begin{array}{cclr}
(k_1, k_1), (k_2, k_2), \ldots, (k_{s_1-1}, k_{s_1-1}), (k_{s_1}, k_{s_1}),\\ (k_1, k_2), (k_2, k_3), \ldots, (k_{s_1-1}, k_{s_1}), (k_{s_1}, k_{1}),\\ 
\vdots\\
(k_1, k_{s_1}), (k_2, k_1), \ldots, (k_{s_1-1}, k_{s_1-2}), (k_{s_1}, k_{s_1-1})
\end{array}\right\} $
in case of $s_1=s_2$,\break

and 
$$\left\{\begin{array}{cclr}
(k_1, k_1), (k_2, k_2), \ldots, (k_{s_1-1}, k_{s_1-1}), (k_{s_1}, k_{s_1}),\\ 
(k_1, k_2), (k_2, k_3), \ldots, (k_{s_1-1}, k_{s_1}), (k_{s_1}, k_{s_1+1}),\\ 
\vdots\\
(k_1, k_{s_2-s_1+1}), (k_2, k_{s_2-s_1+2}), \ldots, (k_{s_1-1}, k_{s_2-1}), (k_{s_1}, k_{s_2})\\
(k_1, k_{s_2-s_1+2}), (k_2, k_{s_2-s_1+3}), \ldots, (k_{s_1-1}, k_{s_2}), (k_{s_1}, k_{1})\\
\vdots\\
(k_1, k_{s_2}), (k_2, k_1), \ldots, (k_{s_1-1}, k_{s_1-2}), (k_{s_1}, k_{s_1-1})\\
\end{array}\right\}$$
if $s_1\lneq s_2$. \newline

By the definition of the tensor product of graphs $G_1, G_2$ we get that vertices 
$(k_{l_1}, k_{l_2}), (k_{l'_1}, k_{l'_2})$ in $G_1\times G_2$ are joined by an edge if, and only if, $k_{l_1}$ and $k_{l'_1}$, $k_{l_2}$ and $k_{l'_2}$ comprise edges in $G_1, G_2$ respectively.
All vertices $k_{l_i}, k_{l'_i}$ are elements of the clique $K_{t_i}$, and so they are joined by an edge if, and only if, $l_i\neq l'_i.$
Therefore vertices writed down in rows of the above sets $K_{t_1}\times K_{t_2}$ are maximal cliques in $G_1\times G_2$.
If $K_{t_1}, K_{t'_1}\in \bigcup_{K_{t_1}}$ and $t_1\neq t'_1$, then there is no any vertex in $K_{t_1}\times K_{t_2}$ forming edges with all elements of some vertex-disjoint maximal clique of $K_{t'_1}\times K_{t'_2}$.\newline
Thus $\bigcup_{t_1=1}^{m_1}\bigcup_{t_2=1}^{m_2}K_{t_1}\times K_{t_2}=V(G_1\times G_2)$ is a sum of $m_1m_2s_2$ vertex-disjoint maximal cliques with $s_1$ elements.\newline 
Applying the induction we get the claim.
\end{proof}
\begin{theorem}\label{kliki.radyka³}
Let $R$ be a finite ring such that $R/J$ is isomorphic to $R_1\times \cdots\times R_n$ and  $V\bigl(G(R_i, \Delta_i)\bigr)$ is a sum of $m_i$ vertex-disjoint maximal cliques $K_i$ with $|K_i|=s_i$ for all $i=1, \ldots, n$ and let $\min\{s_i; i=1, \ldots, n\}=s$. 
There exists a partition of $G(R, \Delta)$ on a sum of $\frac{m_1 \cdots m_ns_1\cdots s_n}{s}|J|$  vertex-disjoint maximal cliques with $\omega\bigl(G(R, \Delta)\bigr)=s$.
 \end{theorem}
\begin{proof}
In view of Theorem \ref{product}, we have $G(R/J, \Delta_J)\simeq G(R_1, \Delta_1)\times G(R_2, \Delta_2)\times \cdots \times G(R_n, \Delta_n)$.\newline 
By Proposition \ref{kliki.produkt}, $V\bigl(G(R/J, \Delta_J)\bigr)$ is a sum of $m_1 \cdots m_ns_1\cdots s_n$ vertex-disjoint maximal cliques with $\omega\bigl(G(R, \Delta)\bigr)=s$. 
Proposition \ref{kliki.iloraz} now yields to desired claim. 
\end{proof}
\begin{cor}
Let $R$ be a ring such that $R/J$ is isomorphic to the direct product of $n$ copies  of $F(q)$. There exists a partition of the distant graph $G(R, \Delta)$ on a sum of $(q+1)^{n-1}|J|$ vertex-disjoint maximal cliques with $\omega\bigl(G(R, \Delta)\bigr) =q+1$. The ring of lower triangular $n\times n$ matrices over the field $F(q)$ is one example of such rings and $|J|=q^{\frac{n^2-n}{2}}$ in this case.
\end{cor}
\begin{theorem}\label{izo.graph}
Let $R, R'$ be finite rings. $G(R, \Delta)$ and $G(R', \Delta ')$ are isomorphic if, and only if, $|R|=|R'|$ and  $R/J=\prod_{i=1}^nR_i,\ R'/J'=\prod_{i=1}^{n}R_{\sigma(i)}$, where $R_i, R_{\sigma(i)}$ are full matrix rings over finite fields, with a permutation $\sigma$ of $\{1, 2, \ldots, n\}$ such that $\alpha_i:R_i\rightarrow R_{\sigma(i)}$ is an isomorphism or an anti-isomorphism.
\end{theorem}
 \begin{proof}
$"\Rightarrow"$ This is straightforward from \cite[Corollary 6.8]{hav3}.\newline
$"\Leftarrow"$  An isomorphism or an anti-isomorphism $\alpha_i:R_i\rightarrow R_{\sigma(i)}$ gives $G(R_i, 
\Delta_i)\simeq G(R_{\sigma(i)}, \Delta_{\sigma(i)})$ for all $i=1, \ldots, n$. Hence $G(R/J, \Delta_J)$ and $G(R'/J', \Delta_{J'}')$ are isomorphic and from the connection between $G(R/J, \Delta_J)$ and $G(R, \Delta)$ we get an isomorphism of $G(R, \Delta)$ and $G(R', \Delta ')$.
\end{proof}

Any finite ring with identity is semiperfect. By the structure theorem of such rings \cite{behrens} $R/J$ is artinian semisimple and idempotents lift modulo $J$.
Hence it has a unique decomposition into a direct product of simple rings:
$$R/J\backsimeq \overline{R_1}\times\overline{R_2}\times\cdots\times\overline{R_m}.$$
According to Theorem \ref{product} we get: $$G(R/J, \Delta_J)\backsimeq\bigotimes_{i=1}^m G(\overline{R_i}, \Delta_i).$$
Any simple ring $\overline{R_i}$ is isomorphic to a full matrix ring $M_{n_i}(q_i)$ over the finite field with $q_i$ elements:
$$G(R/J, \Delta_J)\backsimeq\bigotimes_{i=1}^m G(M_{n_i}(q_i), \Delta_i).$$
It follows then that the description of the projective line over any finite
ring can be based on the projective line over the full matrix ring $\mathbb{P}(M_n(q))$. There is a bijection between $\mathbb{P}(M_n(q))$ and the Grassmannian $\mathscr{G}(n, 2n, q)$, i.e.  the set of all $n$-dimensional subspaces of $V(2n, q)$  \cite[2.4 Theorem.]{spreads}. Consequently, any point  of $\mathbb{P}(M_n(q))$ can be expressed by using of a basis of the corresponding n-dimensional subspace of $V(2n, q)$. The point
\medskip
$$M_n(q)\left(\left[\begin{array}{cclr}
q_{11}&q_{12}&\ldots&q_{1n}\\q_{21}&q_{22}&\ldots&q_{2n}\\ \vdots &\vdots&\vdots&\vdots\\ q_{n1}&q_{n2}&\ldots&q_{nn}
\end{array}\right], \left[\begin{array}{cclr}
q'_{11}&q'_{12}&\ldots&q'_{1n}\\q'_{21}&q'_{22}&\ldots&q'_{2n}\\ \vdots &\vdots&\vdots&\vdots\\ q'_{n1}&q'_{n2}&\ldots&q'_{nn}
\end{array}\right]\right)$$

 corresponds, for instance, to the system of vectors 

$$\begin{array}{cclrcccc}
(q_{11}&q_{12}&\ldots&q_{1n}&q'_{11}&q'_{12}&\ldots&q'_{1n})\\(q_{21}&q_{22}&\ldots&q_{2n}&q'_{21}&q'_{22}&\ldots&q'_{2n})\\ \vdots &\vdots&\vdots&\vdots& \vdots &\vdots&\vdots&\vdots\\ (q_{n1}&q_{n2}&\ldots&q_{nn}&q'_{n1}&q'_{n2}&\ldots&q'_{nn})
\end{array}.$$

The distant graph of the projective line over the full matrix ring $G(M_n(q), \Delta)$, is isomorphic to the graph on $\mathscr{G}(n, 2n, q)$ whose vertex set is $\mathscr{G}(n, 2n, q)$ and whose edges are pairs of complementary subspaces $X, Y \in \mathscr{G}(n, 2n, q)$:
$$\bigwedge_{X, Y \in \mathscr{G}(n, 2n, q)}X\Delta Y \Leftrightarrow X\oplus Y=V(2n, q).$$

Another graph on $\mathscr{G}(n, 2n, q)$ is the well known Grassmann graph, which has the same set of vertices as the distant graph but $X, Y \in \mathscr{G}(n, 2n, q)$ form an edge, whenever both $X$ and $Y$ have codimension  $1$ in $X+Y$, i.e. they are {\sl adjacent} (in symbols: $\sim$ ):
$$\bigwedge_{X, Y \in \mathscr{G}(n, 2n, q)}X\sim Y \Leftrightarrow \dim\big((X+Y)/X\big)=\dim\big((X+Y)/Y\big)=1.$$
$G(M_n(q), \Delta)$ can be described using the notion of the Grassmann graph \cite[Theorem 3.2]{hav 4}. These graphs have been thoroughly investigated by different authors (see for
example \cite{metsch}), however, this special case of the Grassmann graph $\mathscr{G}(n, 2n, q)$ is not characterized.\newline
We can give the number of vertices of $G(M_n(q), \Delta)$ (cf. \cite[p. 920]{design}), i.e. the number of $n$-dimensional subspaces of $V(2n, q)$:
$$|V\bigl(G(M_n(q), \Delta\bigr)|=\frac{(q^{2n}-1)(q^{2n}-q)\ldots(q^{2n}-q^{n-1})}{(q^n-1)(q^n-q)\ldots(q^n-q^{n-1})}=\left[\begin{array}{cclr}
2n \\n
\end{array}\right]_q.$$

The degree of a vertex $v\in G(M_n(q), \Delta)$ is equal to the number of $n$-dimensional subspaces of $V(2n, q)$ that are disjoint to any $n$-dimensional subspace:
$$\deg(\upsilon)=\frac{(q^{2n}-q^n)(q^{2n}-q^{n+1})\ldots(q^{2n}-q^{2n-1})}{(q^n-1)(q^n-q)\ldots(q^n-q^{n-1})}=q^{n^2}.$$
It means that $\deg(v)=|M_n(q)|$ and generally if $v\in G(R, \Delta)$ then $\deg(\upsilon)=|R|$, which is also due to the fact that $GL_2(R)$ acts transitively on $\mathbb{P}(R)$. 
Maximal cliques in $G(M_n(q), \Delta)$ correspond to $n$-spreads in the $2n$-dimensional vector space over $F$. It is known that such a $n$-spread containes $q^n+1$ $n$-dimensional vector subspaces. 
Any partition of the distant graph $G(M_n(q), \Delta)$ on a sum of vertex-disjoint maximal cliques corresponds to 
an $n$-parallelism of the vector space $V(2n, q)$. Therefore and on account of \cite[Theorem 1.]{parallelism}, which has been also proved (independently) by Denniston \cite{packing}, there exists a partition of the distant graph $G(M_2(q), \Delta)$ on a sum of $q^2+q+1$ vertex-disjoint maximal cliques with  $\omega\bigl(G(M_2(q), \Delta)\bigr)=q^2+1$ for any $q$.\break

We pay attention now to the distant graph $G(M_2(2), \Delta)$ which has 35 vertices. 
 \begin{theorem}
The distant graph $G(M_2(2), \Delta)$ has $240$ distinct partitions on a sum of vertex-disjoint maximal cliques. They fall into two conjugacy classes of $120$ each under the action of the linear automorphism group of $G(M_2(2), \Delta)$.
\end{theorem}
\begin{proof}
The proof follows directly from \cite[Theorem 17.5.6 ii]{hir}.
\end{proof} 
We can identify the graph $G\big(M_n(q), \Delta)\big)$ and the corresponding Grassmannian $\mathscr{G}(n, 2n, q)$. Then all automorphisms of the distant graph $G(M_2(2), \Delta)$ are linear or superpositions of linear with the automorphisms defined by duality and annihilator mapping; see \cite{pankov}. Automorphisms of the first type fix the two conjugacy classes of partitions and these of the second type exchange them. \newline
Below we write down one partition from each conjugacy class. In both tables the seven members of the partition are maximal cliques of size five, which are labelled as I, II, \ldots, VII. Thereby each point of  the graph $G(M_2(2), \Delta)$ is described in terms of two basis vectors of its  corresponding subspace in $\mathscr{G}(2,4,2)$.
\begin{table}[h]
\centerline{$\begin{tabular}{|c|c|c|c|c|c|c|}\hline
I&II&III&IV&V&VI&VII\\ \hline
 (0010) & (1001) & (1011) & (1010) & (1010) & (1001) & (1001)\\
(0001) & (0101) & (0100) & (0100) & (0110) & (0100) & (0110)\\ \hline
 (1000) & (1000) & (1000) & (1011) & (1000) & (1011) & (1010) \\
(0100) & (0111) & (0110) & (0111) & (0101) & (0101) & (0111)\\ \hline
(1010) & (1010) & (0010) & (0010) & (0010) & (0110) & (0100) \\
(0101) & (0001) & (0101) & (1101) & (1001) & (0001) & (0001)\\  \hline
(1001) & (1110) & (1010) & (0110) & (1110) & (1000) & (0010) \\
(0111) & (1101) & (1001) & (0101) & (0001) & (0010) & (1100)\\ \hline
(1011) & (0100) & (0001) & (1000) & (0100) & (1100) & (1000) \\
(0110) & (0010) & (1100) & (0001) & (0011) & (0011) & (0011)\\  \hline
\end{tabular}$}
\caption{Partition 1}
\label{table1}
\end{table}
\newpage
\begin{table}[ht]\centerline{$
\begin{tabular}{|c|c|c|c|c|c|c|}\hline
I&II&III&IV&V&VI&VII\\ \hline
 (0010) & (1001) & (1000) & (1000) & (1010) & (1011) & (1010)\\
(0001) & (0101) & (0110) & (0101) & (0100) & (0101) & (0111)\\ \hline
 (1000) & (1011) & (1001) & (1010) & (1011) & (1000) & (1001) \\
(0100) & (0100) & (0100) & (0110) & (0111) & (0111) & (0110)\\ \hline
(1010) & (1110) & (1010) & (0010) & (0110) & (1010) & (1000) \\
(0101) & (1101) & (0001) & (1001) & (0101) & (1001) & (0011)\\ \hline
(1011) & (0110) & (0010) & (1110) & (0010) & (0100) & (0010) \\
(0110) & (0001) & (0101) & (0001) & (1101) & (0010) & (1100)\\ \hline
(1001) & (1000) & (0011) & (0100) & (1000) & (0001) & (0100) \\
(0111) & (0010) & (1100) & (0011) & (0001) & (1100) & (0001)\\ \hline
\end{tabular} $}
\caption{Partition 2}
\label{table2}
\end{table}
We study now cliques formed by vertices of any two maximal cliques of the first partition (Table \ref{table1}). We see that there exists exactly one maximum clique with four elements for any two different maximal cliques. As an example, we show edges formed by vertices of cliques I and II (Figure \ref{najwiêksza.klika}). Edges comprised by vertices of maximum clique are represented by thicker line.
\begin{figure}[h]
\centerline{\includegraphics[width=0.65\textwidth]{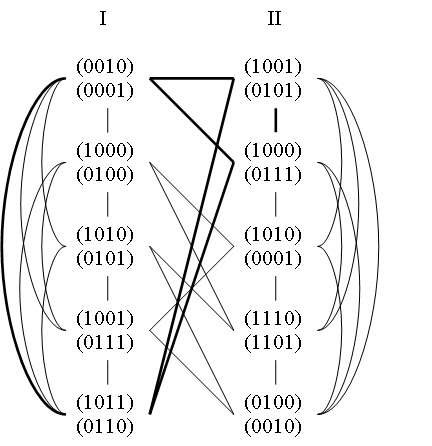}}
\caption{Cliques formed by vertices of cliques I and II}
\label{najwiêksza.klika}
\end{figure}
\newpage
So, for any two of three vertex-disjoint maximal cliques we have one maximum clique  and we checked that three such maximum cliques are of two distinct kinds: either any two of them have one common vertex (Figure \ref{nondisjoint.cliques}) or they are pairwise disjoint (Figure \ref{disjoint.cliques}). The same result can be drawn for the second partition. 
\begin{figure}[h]
\centerline{\includegraphics[width=0.48\textwidth]{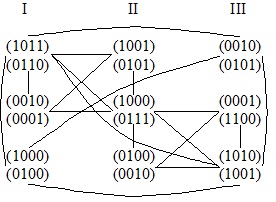}}
\caption{Maximum cliques formed by vertices of cliques I, II and III}
\label{nondisjoint.cliques}
\end{figure}
\begin{figure}[h]
\centerline{\includegraphics[width=0.48\textwidth]{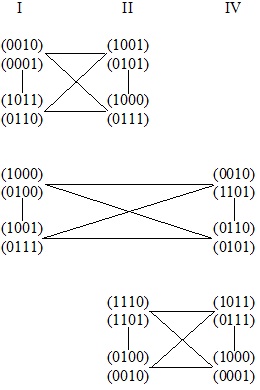}}
\caption{Maximum cliques formed by vertices of  cliques I, II and IV}
\label{disjoint.cliques}
\end{figure}
\newpage
By direct verification we found that vertex-disjoint maximal cliques are points of the projective plane of order $2$. As lines of this plane we take triples of vertex-disjoint maximal cliques of the second kind (Figure \ref{fano}).
\begin{figure}[h]
\centerline{\includegraphics[width=0.55\textwidth]{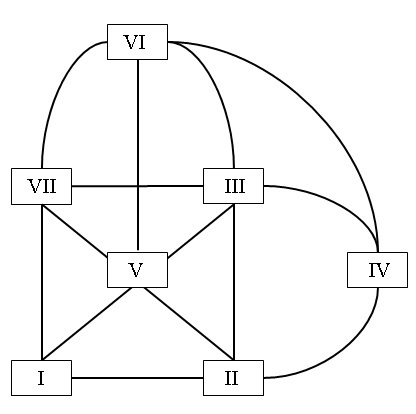}}
\caption{Projective plane of order $2$}
\label{fano}
\end{figure}\newline
Thus we get a simple alternative constraction of the Fano plane described by Hirschfeld in \cite[Theorem 17.5.6]{hir} in projective geometry language.
\newline
There is no proof of the existence of a partition of any graph  $G(M_n(q), \Delta)$. But this problem is well known as an $n$-parallelism in combinatorial design. Sarmiento in \cite{sarmiento} described the partition of the design corresponding to that of  $G(M_3(2), \Delta)$.

\section{The classification of distant graphs}
We start with a characterization of the distant graph of the projective line over any local ring.
\begin{theorem}
\label{local}
Let $R$ be a local ring.
There exists a partition of the distant graph $G(R, \Delta)$ on a sum of $|J|$ vertex-disjoint maximal cliques with  $\omega\bigl(G(R, \Delta)\bigr)=|R/J|+1$.
\end{theorem}
\begin{proof}
If $R$ is local, then $J$ is the maximal ideal of $R$, $R/J$ is a field, and so $G(R/J, \Delta_J)$ is a complete graph with $|R/J|+1$ vertices. According to the connection between $G(R/J, \Delta_J)$ and $G(R, \Delta)$ described in section \ref{construction}, taking into account Remark \ref{ilorazowy} we obtain that vertices $R(a_i, b_i)$ of $G(R, \Delta)$ corresponding to the vertex $R/J(\overline{a}, \overline{b})$ of $G(R/J, \Delta_J)$ are not joined by an edge, while they form an edge with any other vertex of $G(R, \Delta)$. This finishes the proof.
\end{proof}
Let now $v_i^j, u_k^l$ be vertices of $G(R, \Delta)$ and let $V\big(G(R, \Delta)\big), E\big(G(R, \Delta)\big)$ be the sets of vertices and edges of this graph respectively.  We described $G(R, \Delta)$ explicitly in case of a local ring $R$:
$$V\big(G(R, \Delta)\big)=\{v_i^j; i=1, ..., |J|, \ j=1, ..., |R/J|+1\},$$
$$E\big(G(R, \Delta)\big)=\{(v_i^j, u_k^l); j\neq l, \ i,k=1, ..., |J|\}.$$
The sets $\{v_i^j; i=1, ..., |J|\}$, where $j\in\{1,\ldots,|R/J|+1\}$ is
fixed, are maximal anticliques and the sets $\{v_i^j;
j=1,\ldots,|R/J|+1$, where $i\in\{1,\ldots,|J|\}$ is fixed, are maximal
cliques.\newline 

Any finite commutative ring is the direct product of local rings \cite[VI.2]{mcd}. Thus the distant graph of the projective line over any finite commutative ring is known by the above and Theorem \ref{product}.\newline
Every finite ring is isomorphic to the direct product of rings of prime power order \cite[I.1]{mcd}. Hence the distant graph of the projective line over a finite ring can be also described as the tensor product of the distant graphs of the projective lines over rings of prime power order.\newline 
We classify below distant graphs $G(R, \Delta)$, where  $R$ is an indecomposable ring up to order $p^5, p$ prime.
We use some facts and the notations that were established in \cite{p5}. Namely, any finite ring can be represented as $R=S\oplus M$, where $S=\bigoplus_{i=1}^mR_i, \ R_i$ are primary rings and $M$ is a bimodule over the ring $S$. $M$ is also an additive subgroup of $J$, so $M\subset J$ and we thus get $R/J\simeq S/J$.\newline
In the case of $p, \ p^2$ for any $p$ we get complete graphs of order $p+1, p+2$ and the graph $G(R, \Delta)$, where $R$ is local and $|R|=p^2$, $|J|=p$. 
If $|R|=p^3$ then we have three graphs $G(R, \Delta)$ for any $p$: complete graph of order $p^3+1$, the graph where $R$ is local and $|J|=p^2$ and the graph of the projective line over the ring of lower triangular $2\times 2$-matrices over a field $F(p)$, which is a sum of $p^2+p$ vertex-disjoint maximal cliques with  $\omega\bigl(G(R, \Delta)\bigr)=p+1$.
\begin{theorem}
Let $R$ be an indecomposable ring of order $p^4, \ p$ prime. There are exactly five nonisomorphic graphs $G(R, \Delta)$ for any $p$. These are: 
\begin{enumerate}
\item the complete graph of order  $p^4+1$;
\item the graph on the projective line over local ring with $|J|=p^2$;
\item the graph on the projective line over local ring with $|J|=p^3$;
\item the graph which is a sum of $p^3+p^2$ vertex-disjoint maximal cliques with $\omega\bigl(G(R, \Delta)\bigr)=p+1$, $(R/J=F(p)\times F(p))$;
\item $G(M_2(p), \Delta)$ (described in section {\rm \ref{construction}}).
\end{enumerate}
\end{theorem}
\begin{proof}
The proof is straightforward from the classification of rings of order $p^4$ in \cite{p4, p5},  Theorems \ref{product}, \ref{local} and the connection between $G(R, \Delta)$ and $G(R/J, \Delta_J)$.
\end{proof}
\begin{theorem}
Let $R$ be an indecomposable ring of order $p^5, \ p$ prime. There are exactly six nonisomorphic graphs $G(R, \Delta)$ for any $p$. These are: 
\begin{enumerate}
\item \label{1} the complete graph of order  $p^5+1$;
\item\label{2} the graph on the projective line over local ring with $|J|=p^3$;
\item\label{3} the graph on the projective line over local ring with $|J|=p^4$;
\item\label{4} the graph which is a sum of $p^4+p^3$ vertex-disjoint maximal cliques with $\omega\bigl(G(R, \Delta)\bigr)=p+1$, $(R/J=F(p)\times F(p))$;
\item\label{5} the graph which is a sum of $p^4+p^2$ vertex-disjoint maximal cliques with $\omega\bigl(G(R, \Delta)\bigr)=p+1$, $(R/J=F(p)\times F(p^2))$;
\item\label{6} the graph which is a sum of $(p+1)^2$ vertex-disjoint maximal cliques with $\omega\bigl(G(R, \Delta)\bigr)=p+1$, $(R/J=F(p)\times F(p)\times F(p))$;
\end{enumerate}
\end{theorem}
\begin{proof}
The proof for \ref{1}., \ref{2}., \ref{3}. is similar to the proof of the previous theorem. It follows immediately from Theorem \ref{local} and the classification of rings of order $p^5$ in \cite{p5}.\newline
By characterization of the ring part $S$ in representation $R=S\oplus M$ in \cite{p5} we obtain three type of rings with the following ring part:
\begin{enumerate}[(a)]
\item $R_1\times R_2$, where $R_1, R_2$ are
\begin{itemize}
\item proper local rings of order $p^2$, $p^3$ or a field $F(p)$ such that $|R_1||R_2|=p^4$;
\item proper local rings of order $p^2$ or a field $F(p)$ such that $|R_1||R_2|=p^3$;
\item fields $F(p)$.
\end{itemize}
In this case $R/J=F(p)\times F(p)$.
\item $F(p^2)\times F(p)$ or  $F(p)\times F(p^2)$. It is necessary to explain how these two ring parts represent two distinct rings. And that is, they have different module parts: $_{F(p^2)}M_ {F(p)}$ and $_{F(p)}M_ {F(p^2)}$ respectively. For all these rings $R/J=F(p)\times F(p^2)$.
\item $F(p)\times F(p)\times F(p)$ and also $R/J=F(p)\times F(p)\times F(p)$.
\end{enumerate}
The connection  between $G(R, \Delta)$ and $G(R/J, \Delta_J)$ and Theorems \ref{product} and \ref{local} now completes the proof.
\end{proof}

When this paper was finished we became aware of the recent preprint by Silverman (arXiv:1612.08085) which in part is addressed to the same topic.

\newpage


\noindent {University of Warmia and Mazury in Olsztyn\\ Faculty of Mathematics and Computer Science\\ S{\l}oneczna 54 Street\\ 10-710 Olsztyn, Poland\\{\ttfamily edytabartnicka@wp.pl, amatras@uwm.edu.pl}}
\end{document}